\documentclass{article}
\usepackage{amssymb,graphicx,graphics,color,amsopn,amsfonts,amsmath,hyperref}

\newtheorem{theorem}{Theorem}[section]

\newtheorem{cor}[theorem]{Corollary}
\newtheorem{lemma}[theorem]{Lemma}

\newtheorem{claim}[theorem]{Claim}

\newenvironment{conjecture}{\medskip \noindent \textbf{Conjecture:} 
  \centering}{\newline}
\newenvironment{question}{\begin{flushleft}\textbf{Open questions:}\end{flushleft}
	\centering}{}

\newenvironment{proof}{\medskip \noindent{\sc Proof:}}{\quad$\Box$\par\medskip} 
\newenvironment{proof2}[1]{\medskip \noindent{\sc Proof of #1:}}{\quad$\Box$\par\medskip} 

\def\pis{{\pi^*}}

\definecolor{brwn}{RGB}{140, 70, 20}
\definecolor{gren}{RGB}{  0, 140, 10}

\title{Restricted optimal pebbling is NP-hard} 

\author{	László F. Papp\thanks{Department of Computer Science and Information Theory, Budapest University of Technology and Economics, Hungary, lazsa@cs.bme.hu.}}

\begin{document}
\maketitle 

\begin{abstract} Consider a distribution of pebbles on a graph. A pebbling move removes two pebbles from a vertex and place one at an adjacent vertex.
	A vertex is reachable under a pebble distribution if it has a pebble after the application of a sequence of pebbling moves. A pebble distribution
	is solvable if each vertex is reachable under it. The size of a pebble distribution is the total number of pebbles. The optimal
	pebbling number $\pi^*(G)$ is the size of the smallest solvable distribution.
	A $t$-restricted pebble distribution places at most $t$ pebbles at each vertex. The $t$-restricted optimal pebbling number $\pi_t^*(G)$ is the size of the smallest solvable $t$-restricted pebble distribution. 
	We show that deciding whether $\pi^*_2(G)\leq k$ is NP-complete. We prove that $\pi_t^*(G)=\pi^*(G)$ if $\delta(G)\geq \frac{2|V(G)|}{3}-1$ and we show infinitely many graphs which satisfies $\delta(H)\approx \frac{1}{2}|V(H)|$ but $\pi_t^*(H)\neq\pi^*(H)$, where $\delta$ denotes the minimum degree. 
\end{abstract}

\section{Introduction} 
Graph pebbling is a game on graphs initiated by a method of Saks and Lagarias to answer a number-theoretic question of Erd\H os and Lemke, which was successfully carried out by Chung in 1989\cite{chung}. 
 Each graph in this paper is simple. 
 We denote the vertex set and the edge set of graph $G$ by $V(G)$ and $E(G)$, respectively. 
 A \emph{pebble distribution} $D$ on graph $G$ is a function mapping the vertex set to non-negative integers. We can imagine that each vertex $v$ has $D(v)$  pebbles. A \emph{pebbling move} removes two pebbles from a vertex $v$ (having at least two pebbles) and places one on an adjacent vertex $u$. We denote this move by $(v\rightarrow u)$. 
 
If $\sigma$ is a sequence of pebbling moves, then let $D\sigma$ be the distribution which we obtain after the application of 
$\sigma$ to distribtuion $D$.
A sequence of pebbling moves is \emph{executable} under $D$ if for any $\tau$ prefix of $\sigma$ $D\tau$ is a pebble distribution.  In other words, $\sigma$ never removes two pebbles from a vertex which does not have at least two pebbles. A vertex $v$ is \emph{$k$-reachable} under a distribution $D$ if there is an executable sequence of pebbling moves $\sigma$ such that $D\sigma(v)\geq k$. When $k=1$ we write reachable instead of $1$-reachable.
 We say that a distribution $D$ is \emph{solvable} if each vertex of the graph is reachable under $D$. 

Let $|D|$ denotes the total number of pebbles placed on the graph by $D$, so $|D|=\sum_{v\in V(G)}D(v)$. We call this quantity as the \emph{size of} $D$.
A pebble distribution $D$ on a graph $G$ is called \emph{optimal} if it is solvable and
its size is the smallest possible. This size is called the \emph{optimal pebbling number} and denoted by $\pi^*(G)$. 

The optimal pebbling number of several graph families are known. For example exact values were given for paths and cycles \cite{path1} \cite{path2}, ladders \cite{bdcm}, caterpillars \cite{caterpillar}, $m$-ary trees \cite{m-ary} and some staircase graphs\cite{stairs}. 
The values for graphs with diameter smaller than four are also characterized by some easily checkable domination conditions \cite{diam}. However, determining the optimal pebbling number for a given graph is NP-hard \cite{NPhard}.    

In \cite{restricted}, Chelalli \textit{et al.} introduced a new version of pebbling. In this setting, a pebble distribution is  \emph{$t$-restricted} if no vertex has more than $t$ pebbles. 
The \emph{$t$-restricted} optimal pebbling number, denoted by $\pi^*_t$, is the size of the solvable $t$-restricted distribution containing the least number of pebbles. It is easy to see that $\pi_2^*(G)\geq \pi_t^*(G)\geq \pi_{t+1}^*(G)\geq \pi^*(G)$.

A set $S$ is a dominating set of graph $G$ if $S\subseteq V(G)$ and each vertex of $G$ is contained in $S$ or adjacent to an element of $S$. The domination number $\gamma(G)$ of graph $G$ is the size of the smallest domination set of $G$. A function $f:V(G)\rightarrow\{0,1,2\}$ is called as a Roman domination function if each $v\in V(G)$ for which $f(v)=0$ there is a vertex $u$ which is adjacent to $v$ and $f(u)=2$. The weight of a Roman domination function is $\sum_{v \in V(G)} f(v)$. The Roman domination number $\gamma_R(G)$ is the minimum weight of a Roman domination function.
    
Chelalli \text{et al.} proved several bounds on these domination parameters by using the $2$-restricted optimal pebbling number. They showed that $\pi_2^*(G)\leq \gamma_R (G)$ \cite{restricted}. 
These domination parameters are well studied and they are in the limelight. Computing the Roman domination number is a hard task. More specifically, the decision version of the Roman domination number is NP-complete \cite{romdom}

Recently Shiue \emph{et al.} in \cite{gameisland} defined a new version of pebbling called $(d,t)$-pebbling. In this new model a pebble distribution $D$ is called $(d,t)$-solvable if each vertex $v$ can have $t$ pebbles after a sequence of pebbling moves which only remove pebbles from vertices whose distance from $v$ is at most $d$. The optimal $(d,t)$-pebbling number $\pi^*_{(d,t)}(G)$ of $G$ is the minimum size of a $(d,t)$-solvable distribution of $G$. Shiue showed that $\gamma_R (G)=\pi^*_{(1,1)}(G)$ \cite{dtpebbling}, so the Roman domination number is a special case of this generalised pebbling parameter. In a $(1,1)$-solvable distribution no vertex has more than $2$ pebbles, because $2$ pebbles at a vertex are enough to reach the adjacent vertices and we cannot move a pebble further in this setting. 
So to sum up, $\pi^*(G)\leq \pi^*_2(G)\leq \pi^*_{(1,1)}=\gamma_R(G)$.

Therefore the $t$-restricted optimal pebbling number is trapped between two graph parameters whose calculation is an NP-hard task. Somebody may ask that what is the computational complexity of the $2$-restricted optimal pebbling number?

In Section \ref{elso} we show that deciding whether the $t$-restricted optimal pebbling number is at most $k$ is NP-complete for any $t\geq 2$. Note that the $t=1$ case is trivial because $\pi^*_1(G)=|V(G)|$ for all graphs. We prove that $\pi^*(G)=\pi^*\left(G\cdot K_m\right )=\pi^*_t\left(G\cdot K_m\right)$ where $\cdot$ denotes the lexicographic graph product. We use this result to give a reduction from optimal pebbling to $t$-restricted optimal pebbling.


The authors of \cite{restricted} asked the following interesting question: What graphs have the same $2$-restricted optimal pebbling number and optimal pebbling number? We can also ask what kind of properties implies the equality of these two parameters. In Section \ref{masodik} we investigate the role of the minimum degree in this question and show some partial results. 

We prove that that if $\delta(G)\geq \frac{2|V(G)|}{3}-1$, where $\delta(G)$ denotes the minimum degree of $G$, then $\pi_2^*(G)=\pi^*(G)$. 
For any $n$ which satifies $n\equiv 1 \mod 4$ and $n\geq 9$ we present an $n$-vertex graph $H$ such that $\delta(H)=\frac{|V(H)|-3}{2}$ but $\pi_2^*(H)\neq \pi^*(H)$.









\section{Complexity of the t-restricted optimal pebbling number} 

\label{elso}

\begin{lemma}
\label{nincsko}
Let $G$ be a connected graph whose order is at least $2$. $G$ has an optimal distribution $D$ such that for every vertex $v$ which is not 2-reachable under $D$ we have that $D(v)=0$.
\end{lemma}

\begin{proof}
If a vertex is not 2-reachable but it has some pebbles, then it contains exactly one pebble.
Let $D$ be an optimal distribution on graph $G$ and let $S(D)$ be the set of vertices which are not 2-reachable under $D$ but each of them has a pebble.

Consider a vertex $v\in S(D)$. Since the graph is connected it has a neighbor $u$. This vertex can receive a pebble without the usage of the pebble placed at $v$.  
If we remove the pebble placed at $v$ and place it at $u$, then under the obtained distribution $u$ is 2-reachable so $v$ is reachable. The relocated pebble could not be moved under $D$ so if a vertex was $k$-reachable under $D$, then it is also $k$-reachable under the obtained distribution $D'$. Therefore $D'$ is optimal and $S(D')\subseteq S(D)$. Since $v$ is not in $S(D')$ we have that $|S(D')|<|S(D)|$. If we repeat this procedure, then we end up with a distribution $D^*$ such that $S(D^*)=\emptyset$.    
\end{proof}

$G\cdot H$ denotes the lexicographic product of graphs $G$ and $H$, which is well known.  It is defined as follows:
$V(G\cdot H)=V(G)\times V(H)$ and $(g_1,h_1)$ and $(g_2,h_2)$ are adjacent iff either $\{g_1,g_2\}\in E(G)$ or $g_1=g_2$ and $\{h_1,h_2\}\in E(H)$.

To prove our first theorem we are going to use the technique called collapsing. It is invented in \cite{bdcm} and it is used in several pebbling papers. 

Let $G$ and $H$ be simple graphs. Graph $H$ is called as a \emph{quotient of} $G$ if there is a surjective mapping $\phi: V(G)\rightarrow V(H)$ such that $\{h_1,h_2\}\in E(H)$ if and only if $h_1=\phi(g_1)$ and $h_2=\phi(g_2)$ for some $\{g_1,g_2\}\in E(G)$. 
We say that $\phi$ {\em collapses} $G$ to $H$ and if $D$ is a pebble distribution on $G$, then the \emph{collapsed distribution} $\phi(D)$ on $H$ is defined by
$\phi(D)(h)=\sum_{g\in V(G)| \phi(g)=h}D(g)$. 

\begin{lemma}[Collapsing \cite{bdcm}]
If $H$ is a quotient of $G$ then $\pi^*(G)\geq \pi^*(H)$.
\label{collapsing}
\end{lemma}

It is known that the optimal pebbling number of any connected $n$-vertex graph is at most $\left\lceil\frac{2n}{3}\right\rceil$ \cite{bdcm}. Interestingly, the same upper bound holds for the $2$-restricted optimal pebbling number \cite{shiue}.  
Now we are well prepared to state and prove our first theorem.

\begin{theorem}
If $G$ is a connected n-vertex graph, $m\geq \left\lceil\frac{n}{3}\right\rceil $ and $t\geq 2$ then:
$$\pi^*(G)=\pi^*\left(G\cdot K_m\right )=\pi^*_t\left(G\cdot K_m\right).$$
\label{egyenloseg}
\end{theorem}

\begin{proof}
Let $D$ be an optimal pebble distribution of $G$ which satisfies that if a vertex has one pebble then it is 2-reachable under $D$. Lemma \ref{nincsko} guarantees that such a pebble distribution exists.

Since $|D|\leq \left\lceil\frac{2n}{3}\right\rceil$ and $m\geq \left\lceil\frac{n}{3}\right\rceil$ we have that $|D|\leq 2m$. $D(g)=2m$ implies that all pebbles are placed at $g$. Assume that this is not the case, so for each $g$ $D(g)<2m$.

Denote the vertices of $K_m$ with $(0,1,2,\ldots, m-1)$. Let $Q$ be the following pebble distribution on $G\cdot K_m$:

\begin{equation*}
Q\left((g,i)\right)=
\begin{cases} 
1 &\text{if}\ i=0\ \text{and}\ D(g)\ \text{is odd}   \\
2 &\text{if}\ 1\leq i\leq \frac{D(g)}{2}\\
0 &\text{otherwise.}
\end{cases}
\end{equation*}

For each $g$ we have that $\sum_{l=0}^{m-1}Q((g,l))=D(g)$. Therefore $|D|=|Q|$. 
Now we show that $Q$ is a solvable distribution on $G\cdot K_m$. To do this we show that any vertex $(g^*,j)$ is reachable under $Q$. 

If $D(g^*)\geq 2$, then $Q((g^*,1))=2$, therefore $(g^*,j)$ is reachable under $Q$. 

When $D(g^*)=1$ we can use that $g^*$ is 2-reachable under $D$. Therefore in the $D(g^*)\leq 1$ case there is an executable sequence of pebbling moves $\sigma$ whose last move places a pebble at $g^*$. We can assume, that $\sigma$ does not contain unnecessary steps, therefore it does not remove a pebble from $g^*$ and it moves exactly one pebble to $g^*$. 

We construct $\tau$ which is an executable sequence of pebbling moves on the graph $G\cdot K_m$ and its last move places a pebble at $(g^*,j)$.
The lengths of $\sigma$ and $\tau$ will be the same. If $g'\neq g^*$, $\sigma_i=(g\rightarrow g')$ and this pebbling move appears exactly $k$ times among the first $i$ elements of $\sigma$, then let
\begin{equation*}
\tau_i=
\begin{cases}
((g,k)\rightarrow(g',0))\ \text{if}\ k\leq\frac{D(g)}{2}\\
((g,0)\rightarrow(g',0))\ \text{otherwise.}   
\end{cases}
\end{equation*}

Finally, if the last step of $\sigma$ is $\sigma_i=(g\rightarrow g^*)$, then let

\begin{equation*}
\tau_i=
\begin{cases}
((g,k)\rightarrow(g^*,j))\ \text{if}\ k\leq\frac{D(g)}{2}\\
((g,0)\rightarrow(g^*,j))\ \text{otherwise.}   
\end{cases}
\end{equation*}

$\tau$ mimics $\sigma$ in the following way: It gathers the pebbles which are received by $g$ at $(g,0)$. If a $(g\rightarrow g')$ move is applied by $\sigma$ then $\tau$ initially removes pebbles from vertices $(g,k)$, $k\geq 1$ where the pebbles of $g$ were distributed. After all these pebbles have been used $\tau$ uses the pebbles gathered to $(g,0)$. 

For any $l$ it is easy to see that $D\sigma_1\sigma_2\ldots \sigma_l(g)=\sum_{i=0}^{m-1}Q\tau_1\tau_2\ldots\tau_l((g,i))$. These guarantee that $\tau$ never wants to remove pebbles from a vertex which does not have at least two pebbles, thus it is executable.   

If $|D|=D(g)=2m$, then let $Q'$ be the distribution on $G\cdot K_m$ which places $2$ pebbles at each of $(g,i)$ and does not place any pebble at the rest of the vertices. It is easy to prove, by the same reasoning which we have used in the previous case, that $Q'$ is a solvable $2$-restricted distribution of $G\cdot K_m$ and $|Q'|=|D|$.   

So in both cases we have shown a solvable $2$-restricted distribution on $G\cdot K_m$ and the size of these distributions is $|D|$. 
Therefore $\pi^*(G)\geq \pi_2^*(G\cdot K_m)$. We know that $\pi_2^*(G\cdot K_m) \geq \pi^*_t(G\cdot K_m)\geq \pi^*(G\cdot K_m) $.
$G$ is a quotient of $G\cdot K_m$ because $\phi((g,i))=g$ is a proper surjective function. Therefore Lemma \ref{collapsing} gives that 
$\pi^*\left(G\cdot K_m\right)\geq\pi^*(G)$.

Putting everything together we get that 
$\pi^*(G)\geq \pi_2^*(G\cdot K_m)\geq  \pi^*_t(G\cdot K_m)\geq \pi^*(G\cdot K_m)\geq\pi^*(G)$. Since the leftmost and rightmost quantities are the same, all the quantities are equal.
\end{proof}

We consider the following two decision problems: 

\textbf{OPN:}  Given $G$, $k$: Is $\pi^*(G)\leq k$?   

\textbf{ROPN:}  Given $G$, $t\geq 2$, $k$: Is $\pi^*_t(G)\leq k$?  

Milans and Clark proved that OPN is NP-complete \cite{NPhard}. Using our previous theorem we are going to show that ROPN is NP-hard.
We do not give a precise proof that ROPN is in NP, since it is exactly the same as the proof that OPN is in NP. In the next paragraph we give the sketch of the proof. If the reader is interested in the details, then they can be found in \cite{NPhard}. 

The witness for $\pi^*(G)\leq k$ is the following: A solvable $2$-restricted pebble distribution $D$ of size $k$ and for each vertex $v$ a compact transcript of an executable sequence of pebbling moves $\sigma_v$, which satisfy $D\sigma_v(v)\geq 1$. 
This transcription does not contain the order of the pebbling moves, it just contains the total number of each possible moves. From this transcription it is possible to decide in polynomial time whether a compatible executable
sequence of pebbling moves exists or not.    


\begin{lemma}
OPN$\prec$ROPN.
\label{reduction}
\end{lemma}

\begin{proof}
Let $f$ be the function which maps $G$ to $G\cdot K_{|V(G)|}$.
If $G$ is the input of OPN then let the corresponding input of ROPN be $f(G)$. 
$f$ can be calculated in polinomial time. By Lemma \ref{egyenloseg} $\pi^*(G)\leq k$ iff $\pi^*_t(f(G))\leq k$. 
\end{proof}

OPN is NP-complete and ROPN is in NP, therefore Lemma \ref{reduction} implies our main complexity result:

\begin{theorem}
ROPN is NP-complete.
\end{theorem}

\section{Large minimum degree enforces $\pis(G)=\pi_t^*(G)$}

\label{masodik}

The authors of \cite{restricted} asked for a characterization of graphs whose optimal pebbling number and $2$-restricted optimal pebbling number are the same. We think that such a characterization is elaborate. Note that there are many graphs which have this property. For example paths, cycles and complete graphs.

Lemma \ref{egyenloseg} gives infinitely many examples. The graphs $G\cdot K_k$ belongs to the investigated family when $k\geq \left\lceil \frac{|V(G)|}{3}\right\rceil$. The number of vertices in $G\cdot K_k$ is $|V(G)|\cdot k$ and the minimum degree $\delta$ is at least $2k-1$. Therefore the ratio of the minimum degree and the number of vertices of $G\cdot K_k$ is at least $\frac{2}{|V(G)|}$.

Somebody may ask, that does high minimum degree guarantees that $\pi^*_2=\pi^*$? The answer is yes. We need a short lemma to prove this. 

\begin{lemma}
Let $G$ be a graph. If $D$ is a pebble distribution on $G$ such that $D(v)=3$ for a particular vertex $v$ and $D(u)\leq 1$ for any other vertex $u$, then $D$ is not optimal. 
\label{3ko}
\end{lemma}

\begin{proof}

To reach a vertex the third pebble of $v$ is not needed. Any nonempty sequence of pebbling moves, which can be applied to $D$, removes two pebbles from $v$ in the first step. 
Then at each step only one vertex has more than one pebble and every other vertex has at most as many pebbles as initially. When the third pebble of $v$ is used each vertex has at most as many pebbles as it has under $D$. Therefore the subsequence whose first move removes the third pebble of $v$ and contains all the later moves is executable under $D$ and reaches the same vertex as the original sequence of pebbling moves. But it removes only two pebbles from $v$, so one useless pebble remains there. 
\end{proof}

This lemma immediately gives us a new class of graphs where equality holds between the two investigated parameters. 

\begin{cor}
If $\pi^*(G)\leq 3$, then $\pi_2^*(G)=\pi^*(G)$.
\label{max3}
\end{cor}

\begin{figure}

\begin{center}
\input{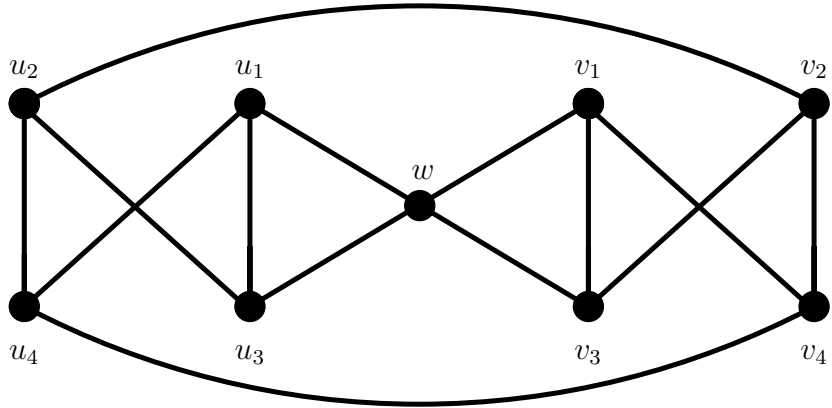}

\end{center}
\caption{The graph $H_4$.}
\label{examplegraph}
\end{figure}


\begin{claim}
Let $G$ be a graph of order $n$. If $\delta(G)\geq \frac{2}{3}n-1$ then $\pi_2^*(G)=\pi^*(G)$.
\label{ketharmad}
\end{claim}

In the proof we are going to use the well known fact that the optimal pebbling number of a diameter two graph is at most $4$\cite{diam}. The neighborhood of a vertex $u$ is the set of vertices which are adjacent to $u$. We denote this set by $N(u)$.

\begin{proof}
We show that $D$ has an optimal distribution which is $2$-restricted. The diameter of $G$ is at most two, because if two vertices are not adjacent, then they share a common neighbor.
If $D$ is an optimal distribution and $|D|\leq 3$, then according to Lemma \ref{3ko} $D$ is 2-restricted. 

If $\pi^*(G)=4$, then we choose two non-adjacent vertices $u$ and $v$ and place two pebbles at each of them. The vertices of $N(u)\cap N(v)$ are $2$-reachable. Therefore each vertex which is adjacent to $\{u,v\}\cup \left(N(u)\cap N(v)\right)$ is reachable. The size of this set is at least $\frac{1}{3}n+2$. Therefore any vertex in $G$ is either in this set or adjacent to it. 
So the constructed distribution is optimal and it is $2$-restricted. 
 \end{proof}


Now we present infinitely many graphs whose optimal pebbling number and $2$-restricted optimal pebbling number are different but their minimum degree is large, where large means that it is almost half of the order.

The construction is the following: Consider two complete graphs of order $m$ where $m$ is even. 
Denote the vertices of the complete graphs by $u_1, u_2,\ldots, u_m$ and   $v_1, v_2,\ldots v_m$, respectively. 
Remove the edges $\{u_i,u_{i+1}\}$ and  $\{v_i,v_{i+1}\}$ for all $i$ which are odd. Note that we have removed a perfect matching. 
Add a vertex $w$ and add all the edges $\{w,u_i\}$ and $\{w,v_i\}$ where $i$ is odd. Add the edges $\{u_i,v_i\}$ for each even $i$. 
Denote this graph by $H_m$. For an example see Figure \ref{examplegraph}.
The minimum degree in this graph is $m-1$ and the number of vertices is $2m+1$.

\begin{claim}
If $m\geq 4$, then $\pi_2^*(H_m)=5$ but $\pi^*(H_m)=4$. 
\label{maxpelda}
\end{claim}

Chelalli \textit{et al.} have characterized the graphs whose 2-restricted optimal pebbling number is at most 4. $\pi_2^*(G)\leq 4$ iff $G$ has vertices $u$ and $v$ such that $u\cup v\cup N(v)$ is a dominating set of $G$ \cite{restricted}. Besides this can be used to prove $\pi^*_2(H_m)>5$, there are a lot of cases. To shorten the proof, we are going to use a different lemma and the weight argument, which was developed by Moews \cite{Moews}. 

The \emph{weight function} of a pebble distribution $D$, which is defined on the vertex set of $G$, is 
$W_D(u)=\sum_{v\in V(G)}D(v)2^{-d(u,v)}$. If the distribution is clear from the context then we simply write $W(u)$. The weight argument claims that if $W_D(u)<1$, then $u$ is not reachable under $D$.

\begin{lemma}
	If  $\pi^*_2(G)\leq 4$ and $G$ is connected, then $G$ has two vertices $x$ and $y$ such that the distribution 
	
	$$D(v)=\begin{cases}
	2 &\text{if } v\in \{x,y\}\\
	0 &\text{otherwise}	
	\end{cases}$$
is solvable. 

\label{specelosztas}
\end{lemma}

\begin{proof}
	Since $\pi^*_2(G)\leq 4$, $G$ has a solvable pebble distribution $D'$ whose size is 4 and it places its pebbles at more than one vertex. 
	If $D'$ put pebbles at exactly two vertices, then $D'$ places $2$ pebbles at each of the two vertices, because of the restriction, and we are done.
	
	If $D'$ places pebbles at three vertices $u,v,w$ then we can assume that $D'(u)=2$ and $D(v)=D(w)=1$. By Lemma \ref{nincsko} we can assume that $u,v$ and $w$ are $2$-reachable vertices. No other vertex is $2$-reachable under $D'$. Since $D'$ is solvable, any other vertex of $G$ need to be adjacent to at least one of $u,v$ and $w$. The pebbles of $u$ are needed in any pebbling sequence and $v$ and $w$ can receive a pebble, so one of $u,v,w$ is adjacent to the other two. If $v,w\in N(u)$, then let 
	$$D(x)=\begin{cases}
		2 & \text{if }x\in\{v,w\}\\
		0 & \text{otherwise.}
	\end{cases}$$ 

    In the other case we can assume WLOG that $v\in N \{u,w\}$ and let 
    	$$D(x)=\begin{cases}
    	2 & \text{if }x\in\{u,w\}\\
    	0 & \text{otherwise.}
    \end{cases}$$ 

    Since $u,v,w$ are $2$-reachable vertices under $D$ and every other vertex is adjacent to one of them, $D$ is a solvable distribution.

	
	If $4$ vertices have pebbles under $D'$, then each of them has exactly one pebble and no pebbling move is allowed, so $|V(G)|=4$. Consider a longest path in $G$ and put $2$ pebbles at each of its end vertices. This is a solvable distribution. Either the other two vertices are contained in a path and they can receive a pebble from a neighbor or the path consists $3$ vertices and the fourth vertex is adjacent to the middle vertex which is $2$-reachable.
\end{proof}

We have all the tools to prove Claim \ref{maxpelda}. In the proof we use the vertex lables $u_i, v_i$ and $w$ in the exact same way as we defined $H_m$.

\begin{proof2}{Claim \ref{maxpelda}} 

Let $D^*$ be the following pebble distribution:
	$$D^*(x)=\begin{cases}
	2 & \text{if }x\in\{v_1,u_1\}\\
	1 & \text{if }x=w\\
	0 & \text{otherwise.}
\end{cases}$$ 

It is easy to see that any vertex except $u_2$ and $v_2$ can be reached by a single pebbling move. The sequence $((u_1\rightarrow u_3),(v_1\rightarrow w),(w\rightarrow u_3),(u_3\rightarrow u_2))$  put a pebble at $u_2$. Vertex $v_2$ can get a pebble in a similar manner, so $D^*$ is solvable and $2$-restricted, thus $\pi_2^*(H_m)\leq 5$. 

Indirectly assume that $\pi^*_2(H_m)\leq 4$. By Lemma \ref{specelosztas} there are vertices $x$ and $y$ in $H_m$ such that the distribution $D(x)=D(y)=2,\ D(z)=0$ if $z\notin \{x,y\}$ is solvable. %
We show by case by case analysis that no such vertices do exist. Without loss of generality, assume that $\sum_{i=1}^m  D(u_i)\leq \sum_{i=1}^m  D(v_i)$.
Let $h,k,l,i$ be integers between $1$ and $\left \lceil\frac{m}{2}\right\rceil$ and calculate the indices of the vertices mod $m$.

\textbf{Case 1. $D(w)=2$:} If $D(v_{2k+1})=2$, then $W(u_{2k+2})=\frac{3}{4}$ so $u_{2k+2}$ is not reachable. If $D(v_{2k})=2$, then $W(u_{2k+2})=1$. Assume that there is a sequence of pebbling moves $\sigma$ which moves a pebble to $u_{2k+2}$, so $D\sigma(u_{2k+2})=1$. Since a pebbling move cannot increase $W(u_{2k+2})$ and it is already $1$, any move contained in $\sigma$ must not decrease $W(u_{2k+2})$. The only pebbling moves which satisfy this condition are $(v_{2k}\rightarrow u_{2k})$ and $(w \rightarrow u_{2j+1})$. But after the application of these moves no more moves are available and $u_{2k+2}$ still does not have a pebble, so no such $\sigma$ exists. Therefore $u_{2k+2}$ is not reachable. 

\textbf{Case 2. $D(w)+\sum_{i=1}^m  D(u_i)=0$:} If $D(v_{2k})=D(v_{2l})=2$ and $k\neq l$, then $W(u_{2k-1})=\frac{3}{4}$, so it is not reachable. If $D(v_{2k+1})=D(v_{2l+1})=2$ and $k\neq l$, then $W(u_{2k+2})=\frac{3}{4}$. If $D(v_{2k+1})=D(v_{2l})=2$, then $W(u_{2l-1})=\frac{3}{4}$.   

\textbf{Case 3. $\sum_{i=1}^m  D(u_i)=2$:} If $D(v_{2k+1})=D(u_{2l+1})=2$, then $W(u_{2k+2})\leq\frac{3}{4}$.

If $D(v_{2k})=D(u_{2l})=2$, then $W(w)=1$. 
The only moves which do not decrease $W(w)$ are $(v_{2k}\rightarrow v_{2h+1})$ and $(u_{2l}\rightarrow u_{2i+1})$. But after the execution of these moves no pebbling move is available and $w$ does not have a pebble so $w$ is not reachable in this case. 

If $D(v_{2k+1})=D(u_{2l})=2$, then $W(u_{2l-1})=1$. The only move of the pebbles of $v_{2k+1}$ which does not decrease $W(u_{2l-1})$ is $(v_{2k+1}\rightarrow w)$. However $W(w)<2$, so the pebble which arrives at $w$ cannot be moved further and the pebbles of $u_{2l}$ are not enough to reach $u_{2l-1}$, so $u_{2l-1}$ is unreachable. 

The subcase $D(u_{2k+1})=D(v_{2l})=2$ is equivalent to the previous subcase due to symmetry. 

We have find an unreachable vertex in each case, therefore no such $D$ exists. 
This contradicts Lemma \ref{specelosztas} and thus $\pi_2^*(H_m)\geq 5$. $D^*$ is a solvable $2$-restricted distribution and $|D^*|=5$, therefore $\pi^*(H_m)=5$.

	Vertices $u_i$ and $v_i$ are neighbors of $w$ if $i$ is odd. When $i$ is even, then $u_i$ has a neighbor $u_j$ where $j$ is odd. 
Therefore the distance between any vertex and $w$ is at most $2$.
Hence placing $4$ pebbles at $w$ is a solvable distribution of size $4$. So $\pi^*(H_m)\leq 4$.

Indirectly assume that $\pi^*(H_m)< 4$ and let $D'$ be an optimal distribution of $G$, thus $|D'|\leq 3$. By lemma \ref{3ko} $D'$ does not place $3$ pebbles at the same vertex, thus it is $2$-restricted solvable distribution, which is a contradiction. 
\end{proof2}

 $\delta(H_m)=\frac{|V(H_m)|-3}{2}$ and $\pi_2^*(H_m)\neq \pi^*(H_m)$. This result and Claim \ref{ketharmad} together imply the existence of the following constant $c$: If $\delta(G)\geq c |V(G)|$, then $\pi_2^*(G)=\pi^*(G)$, but for any $\epsilon>0$ there is a graph $H$ such that $\delta(H)\geq (c-\epsilon)|V(G)|$ and $\pi_2^*(H)\neq\pi^*(H)$. We have tried to determine this constant, but up to this day we did not succeed. Our investigations lead us to state the next conjecture. 
 
\begin{conjecture}
If $\delta(G)\geq \frac{1}{2}|V(G)|$, then $\pi_2^*(G)=\pi^*(G)$.
\end{conjecture}

Note that those graphs whose minimum degree is at least half of their order have diameter at most two. Corollary \ref{max3} implies that it is enough to show that each of these graphs has a solvable $2$-restricted pebble distribution of size $4$. We can use the results of \cite{restricted} or Lemma \ref{specelosztas} to reformulate the previous conjecture as a dominating question.  

\begin{conjecture}
If $\delta(G)\geq \frac{1}{2}|V(G)|$, then there are vertices $u$ and $v$ in $G$ such that $\{u,v\}\cup \left (N(u)\cap N(v) \right)$ is a dominating set.
\end{conjecture}  

\begin{question}
	\begin{enumerate}
		\item Can $\pi_2^*(T)$ can be calculated in polynomial time if $T$ is a tree?
		\item Can $\pi^*(T)$ can be calculated in polynomial time if $T$ is a tree? 
		\item What is the maximum $c$ such that $\delta(G)\geq c |V(G)|$ enforces $\pi^*(G)=\pi_2^*(G)$?  
	\end{enumerate}
\end{question}

\section*{Acknowledgments}
This research was supported by the Ministry of Innovation and
Technology and the National Research, Development and Innovation
Office within the Artificial Intelligence National Laboratory of Hungary.

\end{document}